\documentclass[parskip=half]{scrartcl}

\usepackage{amsmath,mathbbol,amssymb,mathbbol,amsthm}
\usepackage{mathpazo}
\usepackage{mathtools}
\usepackage[shortlabels]{enumitem}
\usepackage[noadjust]{cite}

\newcommand{\embedt}{\beta}

\newcommand{\one}{\mathbf{1}}

\newcommand{\blangle}{\bigl\langle}
\newcommand{\brangle}{\bigr\rangle}

\newcommand{\x}{x}
\newcommand{\eps}{\varepsilon}
\newcommand{\embeds}{\hookrightarrow}
\newcommand{\defn}{\coloneqq}

\renewcommand{\Re}{\operatorname{Re}}
\renewcommand{\Im}{\operatorname{Im}}
\newcommand{\uW}{\underline{W}}

\newcommand{\hinf}{\mathrm{H}^\infty}
\newcommand{\ddmu}{}

\newcommand{\R}{\mathbb{R}}
\newcommand{\N}{\mathbb{N}}
\newcommand{\C}{\mathbb{C}}

\DeclareMathOperator{\tr}{tr}      
\DeclareMathOperator{\dom}{dom}    
\DeclareMathOperator{\dist}{dist}  
\DeclareMathOperator{\diam}{diam}  
\DeclareMathOperator{\supp}{supp}  
\DeclareMathOperator{\sign}{sign}  
\DeclareMathOperator{\dive}{div}   

\newcommand{\forma}{\mathfrak{a}}
\newcommand{\formb}{\mathfrak{b}}
\newcommand{\coa}{a}
\newcommand{\coeffs}{s}
\newcommand{\coefft}{t}
\newcommand{\st}{\colon}
\newcommand{\ud}{d}
\newcommand{\abs}[1]{\lvert#1\rvert}

\usepackage[dvipsnames]{xcolor}
\usepackage[colorlinks=true,allcolors={RoyalBlue}]{hyperref}
\hypersetup{final}
\usepackage[capitalise,nameinlink]{cleveref}

\newtheorem{theorem}{Theorem}[section]
\newtheorem{lemma}[theorem]{Lemma}
\newtheorem{proposition}[theorem]{Proposition}
\newtheorem{corollary}[theorem]{Corollary}

\theoremstyle{definition}
\newtheorem{remark}[theorem]{Remark}

\newtheorem{assumption}[theorem]{Assumption}

\title{On the numerical range of second order elliptic operators
  with mixed boundary conditions in $L^p$}

\author{Ralph Chill \thanks{Institut f\"ur Analysis, Fakult\"at
    Mathematik, TU Dresden, 01062 Dresden,
    Germany. (\texttt{ralph.chill@tu-dresden.de})} \and Hannes
  Meinlschmidt \thanks{Radon Institute for Computational and
    Applied Mathematics (RICAM), Altenberger
    Stra{\ss}e 69, 4040 Linz
    Austria. (\texttt{hannes.meinlschmidt@ricam.oeaw.ac.at})} \and
  Joachim Rehberg \thanks{Weierstra{\ss} Institute for Applied
    Analysis and Statistics (WIAS), Mohrenstra{\ss}e 39, 10117
    Berlin (\texttt{joachim.rehberg@wias-berlin.de})}}

\begin{document}

\maketitle

\begin{abstract}
  We consider second order elliptic operators with real,
  nonsymmetric coefficient functions which are subject to mixed
  boundary conditions. The aim of this paper is to provide
  uniform resolvent estimates for the realizations of these
  operators on $L^p$ in a most direct way and under minimal
  regularity assumptions on the domain. This is analogous to the
  main result in~\cite{CFMP06}. Ultracontractivity of the
  associated semigroups is also considered. All results are for
  two different form domains realizing mixed boundary
  conditions. We further consider the case of Robin- instead of
  classical Neumann boundary conditions and also allow for
  operators inducing dynamic boundary conditions. The results
  are complemented by an intrinsic characterization of elements
  of the form domains inducing mixed boundary
  conditions.\end{abstract}

\section{Introduction}

The regularity of solutions of elliptic or parabolic operators
is a classical subject. Uniform estimates for resolvents of
elliptic operators and for the semigroups generated by them are
central instruments for the study of nonautonomous
linear 
or quasilinear parabolic equations.  Much of the theory is
standard nowadays and treated in many comprehensive books on
parabolic evolution equations; we refer exemplarily
to~\cite[Chapter~II]{Am95},~\cite[Chapter~6.1]{Lu95},~\cite{DeHiPr03} 
or~\cite{PrSi16}.

In this work, we provide uniform resolvent estimates for
the $L^p$-realizations of second order elliptic operators with
real, nonsymmetric coefficient functions posed on bounded
domains in $\R^d$ and subject to mixed
boundary conditions, under minimal regularity assumptions on the
domain. In case of smooth domains and real and symmetric
coefficient functions, such uniform resolvent estimates are
classical~(\cite[Chapter~7.3]{Pa83}) and have been generalized
in~\cite{GrKaRe01} to non-smooth domains and mixed boundary
conditions. Moreover, the case of non-symmetric coefficient
functions has been treated in~\cite{CFMP06} under pure Dirichlet or
pure Neumann- or Robin conditions.

Our main result is a complement to the main result in~\cite{CFMP06}. We give an (optimal) estimate for the half angle
$\theta_p$ of the sector containing the numerical range of the
$L^p$-realization of the elliptic operator. The proof given here
differs from the proof in~\cite{CFMP06} and uses ideas from~\cite{CiMa05}. The estimate for the numerical range immediately
yields resolvent estimates outside the sector with half angle
$\theta_p$, and these estimates stand in a one-to-one
correspondence to the holomorphy of the corresponding semigroup
on a sector with half angle $\frac{\pi}{2} -\theta_p$;
see~\cite[Theorem 1.45]{Ou04} for details. This yields an
$\hinf$ functional calculus, $\mathcal{R}$-sectoriality and
maximal parabolic regularity for the associated operators. It is
known that the obtained resolvent estimates estimates are in
general optimal~\cite{CFMP05,Km02}. We moreover mention
ultracontractivity and associated properties. The results extend
to elliptic operators with mixed Robin- and Dirichlet boundary
conditions, and are also  applied to parabolic evolution
equations with dynamical boundary conditions.

All $L^p$-realizations of elliptic operators are here defined
via sesquilinear forms. It turns out that sectoriality of the
underlying operator in $L^p$, the $\hinf$-functional calculus,
${\mathcal R}$-sectoriality and $L^p$-maximal regularity solely
depend on the properties of the coefficients of the second order
elliptic operator and structural properties of the form domain,
but neither on the regularity of the domain nor on the type of
boundary conditions. On the other hand, ultracontractivity of
the underlying semigroup and a characterisation of the trace
zero property on (parts of) the boundary for elements from the
form domain (see Appendix) solely depend on regularity of the
form domain via properties of the domain in $\R^d$ and the type
of boundary conditions. They are, however, stable under the
passage from mixed Neumann- and Dirichlet boundary conditions to
mixed Robin- and Dirichlet boundary conditions, if the part of
the boundary where Dirichlet boundary conditions are imposed
does not change.

For the sake of readability, we consider only pure second order
operators. Moreover, it would also be possible to deal with
weighted Lebesgue- and associated Sobolev spaces which would
allow for more general and involved differential
operators. Since this is rather involved to combine with mixed
boundary conditions and already incorporated in~\cite{CFMP06} in
the case of non-mixed boundary conditions, we have decided to
not include weighted spaces.

\section{Preliminaries} \label{sec.preliminaries}

Let $\Omega \subseteq \R^d$ be a domain. We do not require
$\Omega$ to be bounded and there are no further regularity
assumptions on $\Omega$ until
Section~\ref{sec:ultr-comp-resolv}.
Let us denote the usual (complex) Lebesgue spaces by
$L^p(\Omega)$ and the corresponding first order Sobolev spaces,
given by all $L^p(\Omega)$ functions whose first-order weak
derivatives are again in $L^p(\Omega)$, by $W^{1,p}(\Omega)$.

\subsection{Form domain}
We first define the considered form domains $V$. These will be
Sobolev spaces incorporating a partially vanishing trace
condition leading to associated evolution equations with mixed
boundary conditions. Let $D\subseteq\partial\Omega$ be a closed
subset of the boundary of $\Omega$, the \emph{D}irichlet
boundary part. In all of the following, let $V$ be either of the
following spaces: 
\begin{equation*}
  V = W^{1,2}_D(\Omega) \defn \overline{\uW^{1,2}_D(\Omega)}^{W^{1,2}}  \quad
  \text{or} \quad V = \widetilde{W^{1,2}_D(\Omega)} \defn
  \overline{C_D^\infty(\Omega)}^{W^{1,2}}
\end{equation*}
where
\begin{align*}
  \uW^{1,2}_D(\Omega)& \defn \bigl\{ u\in W^{1,2}(\Omega) \st
  \dist (\supp u , D) >0 \bigr\}
  \intertext{and}
  C_D^\infty(\Omega) & \defn \bigl\{ u\in C^\infty (\Omega) \st u
  = v|_\Omega \text{ for } v\in C^\infty_c (\R^d ) \text{ with }
  \supp v \cap D = \emptyset \bigr\}.
\end{align*}

Clearly, for $D = \emptyset$, the former $V= W^{1,2}_D(\Omega)$
is just the usual $W^{1,2}(\Omega)$. Note that
$C_c^\infty(\Omega) \subset V$ for both choices of $V$, so
either $V$ is dense in $L^2(\Omega)$. Moreover, both spaces
satisfy the following additional properties:
\begin{enumerate}[label=\protect$(V_{\arabic*})$]
\item $V$ is a sublattice of
  $W^{1,2}(\Omega)$,\label{item:assumption.v-sublattice} that is,
  for every $u\in V$ one has $(\Re u)^+\in V$, and
\item $V$ is stable under the operation
  $u\mapsto (|u|\wedge \one)\sign
  u$.\label{item:assumption.v-proj}
\end{enumerate}
Indeed, these properties are classical for
$V = \widetilde{W^{1,2}_D(\Omega)}$ and for $W^{1,2}(\Omega)$,
see~\cite[Propositions~4.4\&4.11]{Ou04}. For
$V = W^{1,2}_D(\Omega)$ with $D \neq \emptyset$, they follow
from the $W^{1,2}(\Omega)$ case for the dense subspace
$\uW^{1,2}_D(\Omega)$ and then by continuity for the whole
$W^{1,2}_D(\Omega)$. We mention that in
property~\ref{item:assumption.v-proj}, the constant function
$\one$ need not be an element of the form domain $V$.

\begin{remark}
  \label{rem:form-domain-equiv-def}
  The above definitions of $V$ imply a certain, implicit zero
  trace property on $D$ for its
  elements in an abstract way. It is possible to obtain more
  explicit characterisations of zero traces, at least under
  certain regularity assumptions. Indeed, in the appendix, we
  show that under natural, very mild assumptions on $D$ and
  $\partial\Omega$ and the regularity of the relative boundary
  of $D$ within $\partial \Omega$
  (\cref{assumption.v3}), the space $W^{1,2}_D(\Omega)$ can be
  characterized by the set of all $u \in W^{1,2}(\Omega)$ which satisfy
  a Hardy-type inequality w.r.t.\ $D$ or which satisfy
    \begin{equation*}
      \lim_{r\searrow0} \frac1{|B_r(x)|}\int_{B_r(x)\cap\Omega}
      |u| = 0
    \end{equation*}
    for $\mathcal{H}_{d-1}$-a.e.\ $x \in D$. Under stronger
    conditions, $W^{1,2}_D(\Omega)$ and
    $\widetilde{W^{1,2}_D(\Omega)}$ in fact coincide and can
    then be characterized intrinsically as mentioned. We refer to
    e.g.~\cite[Theorem~2.1]{EgTo17}.  (We assume such a setup in
    Section~\ref{sec:an-extension-robin} below, see
    \cref{assumption.s}.) See moreover
    also~\cite[Theorem~8.7~(iii)]{BMMM14} in the context of
    $(\eps,\delta)$-domains.
\end{remark}

\subsection{Coefficient function and form}  Let
$\coa = (\coa_{ij})\in L^\infty(\Omega ; \R^{d\times d})$ be a
real, \emph{uniformly elliptic} coefficient function, that is,
$\Re \langle \coa (x) \xi , \xi\rangle \geq \eta \|\xi \|^2$ for
every $x\in\Omega$, $\xi\in\C^d$ and some ellipticity constant
$\eta >0$. Here $\langle \cdot ,\cdot \rangle$ denotes the usual
Hermitian inner product in $\C^d$. It follows from the
boundedness and the uniform ellipticity that $\coa$ is in
addition \emph{uniformly sectorial}, that is, there exists an
angle $\theta_2 \in [0,\frac{\pi}{2}[$ such that

\begin{equation} \label{one} \langle \coa (x) \xi , \xi \rangle
  \in \overline{\Sigma_{\theta_2}} \quad (x\in\Omega , \,
  \xi\in\C^d),
\end{equation}

where
$\Sigma_\theta = \{ r \, e^{i \varphi} \st r \in
\mathopen]0,\infty[ \text{ and } \varphi \in \mathopen]-\theta,
\theta [\}$ is the open sector of half-angle $\theta$ if
$\theta \in \mathopen]0,\frac{\pi}{2}[$ and
$\Sigma_0 = [0,\infty[$ is the positive real axis. Equivalently,
the sectoriality means that

\begin{equation} \label{oneprime} \bigl|\Im \langle \coa (x) \xi
  , \xi \rangle \bigr| \leq \tan\theta_2 \cdot\, \Re \langle
  \coa (x) \xi , \xi \rangle \quad (x\in\Omega , \, \xi\in\C^d).
\end{equation}




We define the sesquilinear form $\forma \colon W^{1,2}(\Omega)
\times W^{1,2}(\Omega) \to \C$ by
\[
  \forma [u,v] = \int_\Omega \langle a \nabla u , \nabla
  v\rangle \ddmu \quad \quad (u,v\in W^{1,2}(\Omega)).
\]
Due to the properties~\ref{item:assumption.v-sublattice}
and~\ref{item:assumption.v-proj} of $V$, and by the assumptions on the
coefficient function $\coa$, the restriction $\forma_V$ of $\forma$ to $V
\times V$
is a sub-Markovian form. This means that $\forma_V$ is closed,
continuous, and accretive, and the associated operator $A_2$ on
$L^2(\Omega)$ given by
\begin{align*}
  \dom A_2 & \defn \bigl\{ u\in L^2(\Omega)\colon u\in V \text{ and  there exists } f\in L^2(\Omega)\text{ such that}\\
  & \phantom{\defn \{ u\in L^2(\Omega)\colon~} \text{for all } v\in V : \forma [u,v] = \int_\Omega f \bar{v} \ddmu~\bigr\}  , \\
  A_2 u & \defn f ,
\end{align*}
is the negative generator of a positive, analytic, contraction
$C_0$-semigroup $T_2$ on $L^2(\Omega)$ which is in addition
$L^\infty$-contractive, see~\cite[Thms.~1.54,~4.2
and~4.9]{Ou04}.) The semigroup $T_2$ extrapolates consistently
to a positive contraction semigroup $T_p$ on $L^p(\Omega)$ for every
$p\in [1,\infty [$, and the semigroup $T_p$ is analytic if $p>1$
(\cite[Proposition~3.12,~p.56/57\&96]{Ou04}). Denote by $A_p$ the
negative generator of $T_p$.

\begin{remark} \label{rem:sobo-space-properties}
  \begin{enumerate}[(a)]
  \item Both choices for $V$ lead to realizations $A_p$ on
    $L^p(\Omega)$ of the second order elliptic operator
    $-\dive (\coa \nabla \cdot )$ equipped with Dirichlet
    boundary conditions on $D\subseteq\partial\Omega$ and
    Neumann boundary conditions on $\partial\Omega\setminus
    D$. In general, $W^{1,2}_D(\Omega)$ induces a stronger form of
    Neumann conditions on $\partial\Omega\setminus D$ for
    functions in the domain of
    $A_2$. 
  This can be seen for example in the case where $\Omega$ is a
  disc around the origin from which the positive $x$-axis is
  removed to form a slit. Then $u \in \dom A_2$ satisfies
  \begin{align*}
    \partial_{\nu\downarrow}u = \partial_{\nu\uparrow}u & =
    0\quad\text{if}~V = W^{1,2}(\Omega)  \intertext{and}  \bigl[ \partial_{\nu}u\bigr]= \partial_{\nu\downarrow}u
    -\partial_{\nu\uparrow}u &= 0 \quad\text{if}~V = \widetilde{W^{1,2}_D(\Omega)}
  \end{align*}
  along the slit, where the arrows stand for the conormal
  derivatives w.r.t.\ $\coa$ taken from either side.
\item There is the nomenclature \emph{good Neumann boundary
    conditions} for $V = \widetilde{W^{1,2}_D(\Omega)}$ and
  \emph{Neumann boundary conditions} for $V=W^{1,2}_D(\Omega)$,
  see~\cite[Chapter~4]{Ou04}, related to the former space being
  a smaller, i.e., more regular subspace of $W^{1,2}(\Omega)$.
  For example, the form $\forma_V$ with good Neumann boundary
  conditions has the advantage of being a regular Dirichlet
  form, in the sense that $C(\overline{\Omega}) \cap V$ is dense
  in $V$.
\end{enumerate}
\end{remark}

\section{The numerical range} \label{s-numrange}

We next determine a sector which includes the numerical range of
$A_2$. First, a preliminary lemma.

\begin{lemma} \label{lem.core.ap} For every $p\in [1,\infty[$,
  the space $\dom A_2\cap \dom A_p \cap L^\infty(\Omega)$ is a core
  for $A_p$, that is, it is dense in $\dom A_p$ equipped with
  the graph norm.
\end{lemma}

\begin{proof}
  The semigroups $T_2$ and $T_p$ are consistent, that is,
  $T_2 = T_p$ on $L^2(\Omega)\cap L^p(\Omega)$. By taking Laplace
  transforms, $(I+A_2)^{-1} = (I+A_p)^{-1}$ on
  $L^2(\Omega)\cap L^p(\Omega)$. The two resolvents thus also coincide
  on the smaller space $L^1(\Omega) \cap L^\infty(\Omega)$, which is
  dense in $L^p(\Omega)$. The resolvent $(I+A_p)^{-1}$ being an
  isomorphism between $L^p(\Omega)$ and $\dom A_p$ (the latter
  space being equipped with the graph norm), it maps dense
  subspaces to dense subspaces. Since $(I+A_p)^{-1}$ maps
  $L^1(\Omega) \cap L^\infty(\Omega)$ onto a subspace of
  $\dom A_2\cap \dom A_p \cap L^\infty(\Omega)$, it follows that
  the latter space is dense in $\dom A_p$.
\end{proof}

The next lemma paves the road for the actual estimation of the
numerical range.

\begin{lemma} \label{l-justif} For every
  $u\in W^{1,2} (\Omega)\cap L^\infty(\Omega)$ and every
  $\alpha \geq 1$, the functions $|u|^\alpha$ and
  $|u|^{\alpha -1} u$ belong to
  $W^{1,2} (\Omega) \cap L^\infty(\Omega)$ and
  \begin{align*}
    \nabla |u| & = \Re (\frac{\bar{u}}{|u|} \nabla u) , \\
    \nabla |u|^\alpha & = \alpha |u|^{\alpha -1} \nabla |u| , \\
    \nabla (|u|^{\alpha -1} u) & = (\alpha -1) |u|^{\alpha -2} u \nabla |u| + |u|^{\alpha -1} \nabla u .
  \end{align*}
\end{lemma}

\begin{proof}
  For $\nabla|u|$, see~\cite[Proposition~4.4]{Ou04}. The remaining
  assertions follow readily by the chain rule and smooth
  approximation.
\end{proof}

The following result is a central one for this work. It is
contained in the proof of Theorem~1.1 in~\cite{CFMP06}, where
the authors establish an estimate of the angle of analyticity of
the semigroup $T_p$. (Compare with \cref{cor.angle.analyticity} below.) We give here an alternative
proof to the one in~\cite{CFMP06}.

\begin{theorem} \label{thm.numrange.neumann} For every
  $u\in W^{1,2}(\Omega) \cap L^\infty(\Omega)$ and every
  $p\in [2,\infty[$,
  \[
    \forma \bigl[u,|u|^{p-2}u\bigr] \in \Sigma_{\theta_p},
  \]
  where
  \[
    \tan \theta_p = \frac{\sqrt{(p-2)^2 +p^2 \tan^2
        \theta_2}}{2\sqrt{p-1}} ,
  \]
  and $\theta_2$ is as in~\eqref{one} or~\eqref{oneprime}.
\end{theorem}

\begin{proof}
  The case $p=2$ follows immediately from the sectoriality
  assumption on the coefficient function $\coa$ (see~\eqref{one}). So we focus on the case
  $p \in \mathopen]2,\infty[$, here proceeding similarly as in
  the proof of~\cite[Lemma~1]{CiMa05}. Let
  $u\in W^{1,2} (\Omega) \cap L^\infty(\Omega)$, and set
  $v \defn |u|^\frac {p-2}{2}u$. Then, by \cref{l-justif},
  $v$ and the functions $|v| = |u|^{\frac{p}{2}}$,
  $|v|^{\frac{2-p}{p}} v=u$ and
  $|v|^{\frac{p-2}{p}} v = |u|^{p-2} u$ all belong to
  $W^{1,2} (\Omega) \cap L^\infty(\Omega)$. By using the
  identities from \cref{l-justif} one obtains \begingroup
  \allowdisplaybreaks
  \begin{align*}
    \forma_N \bigl[u,\abs{u}^{p-2}u\bigr] & = \int\limits_{\Omega}
    \bigl\langle \coa \nabla u, \nabla ( \abs{u}^{p-2} {u}
    )\bigr\rangle  \ddmu  \\
    & = \int\limits_{\Omega} \Bigl\langle \coa \nabla \bigl (|v|^\frac {2-p}{p}v\bigr ), \nabla \bigl ( |v|^\frac {p-2}{p} v \bigr )\Bigr\rangle \ddmu  \\
    & = \int\limits_{\Omega} \langle \coa \nabla v, \nabla v\rangle \ddmu -\left(1-\tfrac{2}{p}\right)^2 \int\limits_{\Omega} \bigl\langle \coa \nabla |v|,\nabla |v| \bigr\rangle \ddmu \\
    & \phantom{=\ } +\left (1-\tfrac {2}{p} \right ) \left( \int\limits_{\Omega} \Bigl\langle \coa \frac {\overline v}{|v|} \nabla v , \nabla | v|\Bigr\rangle \ddmu  - \int\limits_{\Omega} \Bigl\langle \coa \nabla |v|, \frac {\overline v}{ |v|} \nabla v \Bigr\rangle \ddmu \right).
  \end{align*}
  \endgroup Here we put
  \begin{equation} \label{e-split} \phi\defn \Re \bigl (\frac
    {\overline v}{|v|}\nabla v \bigr ) = \nabla |v| \quad
    \text{and} \quad \psi \defn \Im \bigl (\frac {\overline
      v}{|v|}\nabla v\bigr ).
  \end{equation}
  Then
  \begin{equation*}
    \int\limits_{\Omega} \bigl\langle \coa \nabla
    v, \nabla v \bigr\rangle \ddmu  = \int\limits_{\Omega}
    \Bigl\langle \coa \frac {\overline v}{|v|}\nabla
      v, \frac {\overline v}{|v|} \nabla
      v \Bigr\rangle \ddmu  = \int\limits_{\Omega} \bigl\langle
    \coa(\phi +i \psi), \phi +i \psi \bigr\rangle \ddmu ,
  \end{equation*}
  and therefore \begingroup \allowdisplaybreaks
  \begin{align*}
    \forma\bigl[u, \abs{u}^{p-2}u \bigr] & =\int\limits_{\Omega}
    \bigl\langle \coa (\phi+i \psi), \phi+i \psi \bigr\rangle \ddmu -
    \left(1 -\tfrac {2}{p}\right)^2 \int\limits_{\Omega} \langle \coa
    \phi, \phi \rangle \ddmu \\ 
    & \phantom{=\ } \shoveright{ +\left (1-\tfrac {2}{p}
      \right) \left( \int\limits_{\Omega} \bigl\langle \coa (\phi + i
        \psi) , \phi \bigr\rangle \ddmu - \int\limits_{\Omega}
        \bigl\langle \coa\phi , \phi+ i \psi \bigr\rangle \ddmu
      \right)} \\ 
    & = \left ( 1 -\bigl ( 1- \tfrac {2}{p}\bigr )^2 \right
    ) \int \limits_{\Omega} \langle \coa \phi, \phi \rangle \ddmu +
    \int\limits_{\Omega} \langle \coa \psi,\psi
    \rangle \ddmu \\
    & \phantom{=\ } + \frac {2i}{p'} \int\limits_{\Omega} \langle \coa \psi, \phi \rangle \ddmu
    - \frac {2i}{p} \int\limits_{\Omega} \langle \coa\phi , \psi
    \rangle \ddmu .
  \end{align*}
  \endgroup Decomposing $\coa$ into its symmetric and
  antisymmetric part,
  \begin{equation*}
    \coeffs \defn  \frac{\coa+\coa^*}2, \quad \text{and} \quad \coefft \defn  \frac{\coa-\coa^*}2,
  \end{equation*}  
  and noting that
  $1 -\bigl ( 1- \frac {2}{p}\bigr )^2=\frac{4}{p \, p'}$, we
  thus obtain
  \begin{align} \label{e-Real} \notag \Re
    \forma\bigl[u,u\abs{u}^{p-2}\bigl] & = \frac{4}{p \, p'}
    \int \limits_{\Omega} \left[ \langle \coa \phi, \phi \rangle
      +
      \langle \coa \psi,\psi    \rangle \right] \ddmu\\
    & = \frac{4}{p \, p'} \int \limits_{\Omega} \left[ \bigl\|
      \coeffs^{\frac12} \phi\bigr\|^2 + \bigl\|
      \coeffs^{\frac12}\psi\bigr\|^2 \right] \ddmu
    \intertext{and} \notag
    \Im \forma\bigl[u,u\abs{u}^{p-2}\bigl] &=  \frac{1}{p'} \int\limits_{\Omega} \langle \coa \psi, \phi \rangle \ddmu - \frac{1}{p} \int\limits_{\Omega} \langle \coa\phi , \psi \rangle \ddmu \\
    & = \left(1-\tfrac2p\right) \int\limits_{\Omega}
    \bigl\langle \coeffs \psi,\phi \bigr\rangle \ddmu +
    \int\limits_{\Omega} \bigl\langle \coefft \psi,\phi
    \bigr\rangle \ddmu . \label{e-Imag}
  \end{align}
  Hence
  \begin{align*}
    \Im \forma\bigl[u,u\abs{u}^{p-2}\bigl]& =
    \left(1-\tfrac2p\right) \int\limits_{\Omega} 
    \Bigl\langle \coeffs^{\frac12}\,\psi,\coeffs^{\frac12}\,\phi\Bigr\rangle \ddmu + \int\limits_{\Omega} \Bigl\langle \coeffs^{-\frac12}\,\coefft\,\coeffs^{-\frac12}\,
    \coeffs^{\frac12}\,\psi,\coeffs^{\frac12}\,\phi\Bigr\rangle \ddmu \\ 
    & = \int\limits_{\Omega} \left\langle \left[\bigl(1-\tfrac2p\bigr) I + \coeffs^{-\frac12}\,\coefft\,\coeffs^{-\frac12}\, \right]\coeffs^{\frac12}\,\psi,\coeffs^{\frac12}\,\phi \right\rangle \ddmu \\
    & \leq \frac12 \int\limits_{\Omega} \left\|\bigl(1-\tfrac2p\bigr) I + \coeffs^{-\frac12}\,\coefft\,\coeffs^{-\frac12}\, \right\| \left(\bigl\|\coeffs^{\frac12}\,\psi\bigr\|^2 +\bigl\|\coeffs^{\frac12}\,\phi\|^2\right) \ddmu.
  \end{align*}
  Since $\coefft$ is skew-symmetric, so is
  $\coeffs^{-\frac12}\coefft\coeffs^{-\frac12}$, and one gets
  \begin{equation*}
    \left\|\bigl(1-\tfrac2p\bigr)
      I +
      \coeffs^{-\frac12}\,\coefft\,\coeffs^{-\frac12}\,
    \right\| = \sqrt{(1-\tfrac2p)^2 + \bigl\|\coeffs^{-\frac12}\,\coefft\,\coeffs^{-\frac12}\bigr\|^2}.
  \end{equation*}
   
  Now, by the choice of the angle $\theta_2$ (see especially the
  estimate~\eqref{oneprime}),
  \begin{align*}
    |\bigl\langle \coeffs^{-\frac12}\coefft\coeffs^{-\frac12}\coeffs^{\frac12}\Re\xi,\coeffs^{\frac12}\Im\xi\bigr\rangle|  & = | \langle \coefft \Re\xi,\Im\xi\rangle | \\
    & = \frac12 \, | \Im \langle \coa \xi,\xi\rangle | \\
    & \leq \frac{\tan\theta_2}{2} \, \Re\langle \coa\xi,\xi\rangle \\
    & = \frac{\tan\theta_2}{2} \, \left[ \bigl\|\coeffs^{\frac12}\Re\xi\bigr\|^2 + \bigl\| \coeffs^{\frac12}\Im\xi\bigr\|^2 \right] .
  \end{align*}
  This implies
  $\|\coeffs^{-\frac12}\,\coefft\,\coeffs^{-\frac12}\| \le
  \tan\theta_2$, which together with the preceding estimate
  actually yields the claim.
\end{proof}

\begin{remark} \label{r-anleihe}
  \begin{enumerate}[(a)]
  \item The ideas of introducing the function $v$ and of using
    the splitting in~\eqref{e-split} in the proof above are
    taken from~\cite{CiMa05}, while the use of the operator
    $\coeffs^{- \frac {1}{2}} \coefft \coeffs^{- \frac {1}{2}}$
    is borrowed from~\cite{CFMP06}.
  \item It should not come as a surprise that the evaluation of
    the expression
    \[
      \frac {1}{p'} \langle \coa(x) \xi, \chi \rangle - \frac
      {1}{p} \langle \coa(x) \chi, \xi \rangle =
      \left(1-\tfrac2p\right) \bigl\langle \coeffs(x) \xi,\chi
      \bigr\rangle + \bigl\langle \coefft(x)\xi,\chi\bigr\rangle
    \]
    for $\xi$, $\chi \in \R^d$ and $x\in\Omega$ plays a crucial
    role, cf.~\eqref{e-Imag}.  It is an artefact of
    \begin{equation} \label{e-artefactof}
      \blangle \coa(x) (\chi + i \xi),\tfrac {1}{p'} \chi + \tfrac
      {i}{p} \xi \brangle, 
    \end{equation}
    namely its imaginary part, thanks to the fact that the
    coefficient function is supposed to be \emph{real}
    throughout this work. The expression in~\eqref{e-artefactof}
    has turned out to be very important in the case of complex
    coefficients; we refer to~\cite{CaDr19} and~\cite{EHRT19}.
  \end{enumerate}
\end{remark}
  
Let $A$ be a closed, linear operator on a Banach space $X$. The
\emph{numerical range} of this operator is the set

\[
  w(A) \defn \bigl\{ u^*(Au) \st u\in\dom{A}, \, \|u\|_X =1
  \text{ and } u^*\in J(u) \bigr\} ,
\]

where $J$ is the following, \emph{a priori} set-valued duality
map:

\[
  J(u) \defn \bigl\{ u^* \in X^* \st \|u^*\|_{X^*} = 1 \text{
    and } u^*(u) = \|u\|_X \bigr\} .
\]

But if $X = L^p(\Omega)$ for $p\in \mathopen]1,\infty[$ and
$\| u\|_{L^p(\Omega)} = 1$, then $J(u)$ contains only the element
$u^* \cong |u|^{p-2}u$.

We use \cref{thm.numrange.neumann} to determine the
numerical range for the operators $A_p$ associated to the form
$\forma_V$.

\begin{theorem} \label{t-numrange-dissipative} Let
  $p\in [2,\infty[$. Then the numerical range $w(A_p)$ of the
  operator $A_p$ is contained in the closed sector
  $\overline {\Sigma_{\theta_p}}$, where
  \[
    \tan \theta_p = \frac{\sqrt{(p-2)^2 +p^2 \tan^2
        \theta_2}}{2\sqrt{p-1}}
  \]
  with $\theta_2$ as in~\eqref{one}.
\end{theorem}

\begin{proof}
  Let $u\in \dom{A_p} \cap \dom{A_2}\cap L^\infty(\Omega)$
  with $\|u\|_{L^p(\Omega)} = 1$. We show that $|u|^{p-2}u \in V$. Since
  $\dom A_2 \subset V$, we have $u \in V \cap L^\infty(\Omega)$.
  \begin{enumerate}[(a)]
  \item Let first $V = W^{1,2}_D(\Omega)$. Then there exists a
    sequence $(u_n) \subset \uW^{1,2}_D(\Omega)$ such that
    $u_n \to u$ in $W^{1,2}(\Omega)$. Thus, up to a subsequence,
    $u_n \to u$ pointwise almost everywhere. Due to
    $u \in V \cap L^\infty(\Omega)$, we can arrange that the
    approximating sequence is uniformly bounded in
    $L^\infty(\Omega)$,
    $\|u_n\|_{L^\infty(\Omega)} \leq \|u\|_{L^\infty(\Omega)} +
    1$. Since the supports are unchanged,
    $(|u_n|^{p-2}u_n)\subseteq \uW^{1,2}_D(\Omega) \cap
    L^\infty(\Omega)$ and the sequence is uniformly bounded in
    $W^{1,2}(\Omega)$, recall \cref{l-justif}. Thus
    $|u_n|^{p-2}u_n \rightharpoonup |u|^{p-2}u$ in
    $W^{1,2}(\Omega)$ along a subsequence. This implies
    $|u|^{p-2}u \in V$.
  \item Consider next $V =
    \widetilde{W^{1,2}_D(\Omega)}$. Again, there exists a
    sequence $(u_n) \subset C^\infty_D(\Omega)$ such that
    $u_n \to u$ in $W^{1,2}(\Omega)$. As before, it follows that
    $(|u_n|^{p-2}u_n) \subset \uW^{1,2}_D(\Omega) \cap
    L^\infty(\Omega)$ and
    $|u_n|^{p-2}u_n \rightharpoonup |u|^{p-2}u$ in
    $W^{1,2}(\Omega)$ along a subsequence. It remains to show
    that in fact $(|u_n|^{p-2}u_n) \subset V$. Let $n$ be
    fixed. By construction, there is a function $v \in
    C_c^\infty(\R^d)$ such that $v_{|\Omega} = u_n$ and $\supp v
    \cap D = \emptyset$. Now choose a mollifier family
    $(\phi_k)$ and let $v_k \defn |v|^{p-2}v * \phi_k$. Then $v_k \in
    C_c^\infty(\R^d)$ and, for $k$ large enough, $\supp v_k \cap
    D = \emptyset$. Moreover, ${v_k}_{|\Omega} \to |u_n|^{p-2}u_n$
    in $W^{1,2}(\Omega)$. Hence $|u_n|^{p-2}u_n \in V$.
  \end{enumerate}
  Now, with $|u|^{p-2}u \in V$, we finally have
  \[
    u^*(A_pu) = \int_\Omega (A_pu)|u|^{p-2} \overline{u} \ddmu
    = \int_\Omega (A_2u)|u|^{p-2} \overline{u} \ddmu
    = \forma (u,|u|^{p-2} u), 
  \]
  so that, by \cref{thm.numrange.neumann},
  $u^*(A_pu) \in \overline{\Sigma_{\theta_p}}$. The set
  $\dom{A_p} \cap \dom{A_2}\cap L^\infty(\Omega)$ being a core
  for ${A_p}$ by \cref{lem.core.ap}, the claim follows from
  an approximation argument.
\end{proof}

\begin{remark} \label{r-symm} Interestingly, the above
  calculations for the nonsymmetric coefficient function $a$
  also reproduce the estimates for the numerical range in case
  of a \emph{symmetric} coefficient function,
  see~\cite[Theorem~3.9]{Ou04}. In this case, $\theta_2 = 0$,
  and hence $\tan\theta_p = \frac{p-2}{2\sqrt{p-1}}$.
\end{remark}

From \cref{t-numrange-dissipative} we immediately deduce
several corollaries in a standard way; compare with~\cite[Ch.~1,
Theorem 3.9]{Pa83}.

\begin{corollary} \label{cor.resolvent.estimate} For every
  $p\in \mathopen]1,\infty[$ the spectrum of $A_p$ is contained
  in the closed sector $\overline{\Sigma_{\theta_p}}$ and, for
  every $z\in\C\setminus \overline{\Sigma_{\theta_p}}$,
  \begin{equation} \label{e-pazy} \bigl\|(z-A_p )^{-1}
    \bigr\|_{\mathcal L(L^p(\Omega))} \le \frac{1}{\dist
      (z,\Sigma_{\theta_p} )}
  \end{equation}
  with $\theta_p$ as in \cref{t-numrange-dissipative}.
\end{corollary}

\begin{proof}
  Let $p\in [2,\infty[$. By \cref{t-numrange-dissipative}, for every
  $z\in\C\setminus \overline{\Sigma_{\theta_p}}$ and every
  $u\in \dom{A_p}$ with $\| u\|_{L^p(\Omega)} =1$,
  \begin{align*}
    \| (z-A_p) u \|_{L^p(\Omega)} & = \| (z-A_p)u\|_{L^p(\Omega)} \, \|
    u^*\|_{{L^{p}(\Omega)}^*} \\ 
    & \geq 
    \bigl| z\, u^*(u) - u^*(A_p u)\bigr| \\
    & = \bigl| z - u^*(A_p u)\bigr| \\
    & \geq \dist (z,\Sigma_{\theta_p}) \, \| u\|_{L^p(\Omega)} .
  \end{align*}
  This inequality shows that $z-A_p$ is injective and has closed
  range.  Since $-1\in\varrho (A_p)$, a connectedness argument
  yields
  $\C\setminus \overline{\Sigma_{\theta_p}} \subseteq \varrho
  (A_p)$, and then the resolvent estimate follows from the above
  estimate. The case $p\in \mathopen]1,2[$ follows by duality.
\end{proof}

\begin{corollary} \label{cor.angle.analyticity} For every
  $p\in \mathopen]1,\infty[$, the semigroup generated by $-A_p$
  extends to an analytic contraction semigroup on the sector
  $\Sigma_{\frac{\pi}{2}-\theta_p}$, where $\theta_p$ is as in
  \cref{t-numrange-dissipative}.
\end{corollary}

\begin{proof}
  The claim for $p\geq 2$ follows from \cref{cor.resolvent.estimate} and the Lumer-Phillips theorem
  (see~\cite[Theorem~1.54]{Ou04}), and the case $p\leq 2$ follows
  by duality.
\end{proof}

\begin{remark}
  It was already observed in~\cite{CFMP06} that the angle
  $\theta_p$ in the foregoing corollary is optimal. Therefore,
  also \cref{t-numrange-dissipative} and \cref{cor.resolvent.estimate} above are optimal as far as the
  angle is concerned. An example showing the optimality is
  provided by the Ornstein-Uhlenbeck semigroup on the weighted space
  $L^p(\R^d;\mu)$, where $\mu$ is the associated invariant
  Gaussian measure (see~\cite{CFMP05}).
\end{remark}

\begin{corollary}[{\cite[Corollary~10.16]{KuWe04}}]
  \label{cor:hinfty}
  For every $p\in \mathopen]1,\infty[$, the operator $A_p$ has a
  bounded $\hinf$-functional calculus on a sector of angle
  $<\frac{\pi}{2}$.
\end{corollary}

\begin{remark}
  Regarding \cref{cor:hinfty}, see also~\cite{KaWe01}. If $\theta_p(\hinf) $ denotes the optimal (so,
  smallest) angle for the $\hinf$-functional calculus, then, by~\cite[Corollary~10.12]{KuWe04}, $\theta_2({\hinf}) \leq \theta_2$,
  and by~\cite[Theorem~12.8]{KuWe04},
  $\theta_p({\hinf}) = \theta_p({\mathcal R})$, where the latter
  is the optimal angle of ${\mathcal R}$-sectoriality. For
  $\Omega = \R^d$ it follows from~\cite[Theorem~14.4]{KuWe04} that
  $\theta_p({\hinf}) \leq \theta_p$, and the previous remark then
  again shows that this estimate is optimal.
\end{remark}

From \cref{cor:hinfty}, we also immediately obtain
maximal $L^q$ regularity for the operators $A_p$. We refer to~\cite[Theorem~1.11]{KuWe04}, or to~\cite{La87} for a different
approach. Let us emphasize that there is no regularity
requirement on $\Omega$.

\begin{corollary}
  \label{cor:max-eg}
  For every $1< p$, $q<\infty$, the operator $A_p$ has
  $L^q$-maximal regularity.
\end{corollary}

\section{Ultracontractivity and compact resolvents}
\label{sec:ultr-comp-resolv}
We next consider ultracontractivity of the semigroups $T_p$
generated by $-A_p$ and associated properties. This requires an
assumption on $\Omega$, which is as follows:

\begin{assumption} \label{ass:sobolev-embed} The form domain $V$
  embeds continuously into $L^\embedt(\Omega)$ for some
  $\embedt >2$.
\end{assumption}

In fact, \cref{ass:sobolev-embed} is equivalently an
assumption on ultracontractivity of the semigroups $T_p$
generated by $-A_p$:

\begin{proposition}[{\cite[Theorem~7.3.2]{Ar04}}]
  \label{prop:ultracontract}
  \cref{ass:sobolev-embed} holds true if and
    only if the consistent semigroup family $T_p$ generated by
  $-A_p$ is \emph{ultracontractive}, that is, for all
  $1 \le p < q \le \infty$ there exists a constant $c>0$ such
  that
  \begin{equation} \label{e-ultrac} \|T_p(t)\|_{\mathcal
      L(L^p(\Omega) \to L^q(\Omega))} \le c t^{-\frac
      {\embedt}{\embedt-2}(\frac {1}{p} - \frac {1}{q})} \quad
    (0 < t \le 1).
  \end{equation}
\end{proposition}

Note that by a scaling argument we necessarily have
$\embedt \leq 2^\star \defn \frac{2d}{d-2}$ in
\cref{ass:sobolev-embed}, the first-order Sobolev
exponent associated to $2$. 
In this case, in the exponent
in~\eqref{e-ultrac}, $\frac{\beta}{\beta-2} = \frac{d}2$.

\begin{corollary}
  \label{cor:compact-resolv}
  Suppose that \cref{ass:sobolev-embed} holds true and
  that $|\Omega|< \infty$. Then the following holds true for
  $p \in \mathopen]1,\infty[$:
  \begin{enumerate}[(a)]
  \item The embedding $V\hookrightarrow L^2(\Omega)$ is
    compact.\label{cor:compact-resolv-embed}
  \item The resolvents $(z + A_p)^{-1}$ are compact
    operators on $L^p(\Omega)$ for every
    $z \in \varrho(-A_2)$.\label{cor:compact-resolv-compact}
  \item The semigroups $T_p(t)$ are compact operators on
    $L^p(\Omega)$ for every
    $t>0$.\label{cor:compact-resolv-semigroup}
  \item $\sigma(A_2) = \sigma(A_p)$ and the spectral projections
    corresponding to the nonzero eigenvalues are independent of
    $p$.\label{cor:compact-resolv-spectrum}
  \end{enumerate}
\end{corollary}

\begin{proof}~\ref{cor:compact-resolv-embed} follows from
  $V \embeds L^\embedt(\Omega)$ as in~\cite[Lemma~7.1]{Dan02}. Thus,
  $(\lambda + A_2)^{-1}$ is a compact operator on $L^2(\Omega)$. By
  compactness propagation via interpolation as
  in~\cite[Theorem~1.6.1]{Da89}, $(\lambda + A_p)^{-1}$ is compact
  for every $\lambda \in \varrho(-A_2)$, which
  is~\ref{cor:compact-resolv-compact}. Ultracontractivity
  implies that $T_2(t)$ is a Hilbert-Schmidt integral operator
  and thus compact on $L^2(\Omega)$ for $t>0$. Thus,~\ref{cor:compact-resolv-semigroup} can be seen from factoring
  $T_p(t)$ through $L^2(\Omega)$, see~\cite[Proposition~7.3.3]{Ar04}.
  Finally,~\ref{cor:compact-resolv-spectrum}
  is~\cite[Corollary~1.6.2]{Da89}.
\end{proof}

\begin{figure}[ht]
  \centering
  \includegraphics[width=0.35\textwidth]{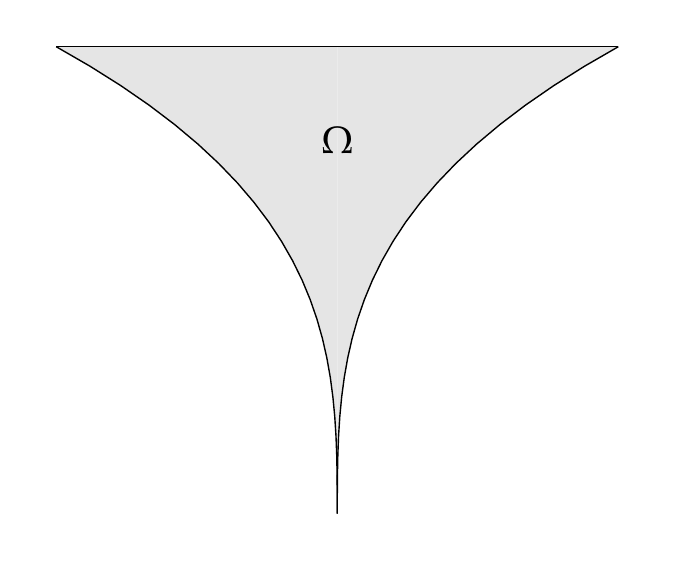}
  \caption{Example of a (non-Sobolev-extension) domain $\Omega$
    satisfying \cref{ass:sobolev-embed}}
  \label{fig:example-domain}
\end{figure}

\begin{remark} \label{rem:domain-cusps-extension} In the case
  $V = W^{1,2}(\Omega)$, so the largest of the form domains
  considered in this work,
  \cref{ass:sobolev-embed} is exhaustively discussed
  in~\cite[Section~6.3.4]{Mz11}. See
  Figure~\ref{fig:example-domain} for the exemplary, two
  dimensional domain
  $\Omega = \{ x \in \R^2 \colon 0<\x_2<1,\, |\x_1| \leq
  \x_2^3\}$ which satisfies \cref{ass:sobolev-embed}
  for $\embedt\leq4$~(\cite{MzPo06}). As visible there, such a
  domain $\Omega$ may have outward cusps, hence it need be
  neither a $d$-set (see~\eqref{eq:d-1} below) nor a
  \emph{homogeneous space}
  (see~\cite[Section~2]{CoWe77}). Therefore $\Omega$ will in
  general not admit a continuous linear extension operator
  $E\colon W^{1,2}(\Omega) \to W^{1,2}(\R^d)$ such that
  $(Eu)_{|\Omega} = u$~(\cite{HjKoTu08}). There may however be
  an continuous linear extension operator $V \to W^{1,2}(\R^d)$,
  see e.g.~\cite[Sect.~6]{EgHDRe15}. The existence of either
  extension operator would imply the optimal $\embedt$ in
  \cref{ass:sobolev-embed}.  Note moreover that there
  might exist bounded extension operators
  $W^{1,2}(\Omega) \to W^{1,r}(\R^d)$ for $r<2$ or even
  $W^{1,2}(\Omega) \to W^{\alpha,2}(\R^d)$ for $\alpha < 1$ for domains
  satisfying \cref{ass:sobolev-embed}.  Conversely,
  for $r$ or $\alpha$ sufficiently large, the existence such an
  extension operator would imply
  \cref{ass:sobolev-embed}. We refer to~\cite{Uk19}
  and the references therein.
\end{remark}

\section{An extension to Robin and dynamical boundary
  conditions}\label{sec:an-extension-robin}

We next show how the generality of the foregoing results, in
particular \cref{thm.numrange.neumann}, can be used to
obtain uniform resolvent estimates for differential operators
attached to more sophisticated problems. To this end, we need
some regularity assumption on $\Omega$ and the boundary part $D$
in order to have a well defined trace-type operator. We
assume that $\Omega$ is bounded throughout this section. The
regularity assumption is as follows.

\begin{assumption} \label{assumption.s}
  \begin{enumerate}[(i)]
  \item For every point
    $x \in \Gamma \defn \overline{\partial\Omega\setminus D}$,
    there is an open neighbourhood $U_x$ of $x$ such that
    $U_x \cap \Omega$ is connected and there exists a continuous
    linear extension operator
    $E \colon W^{1,2}(U_x \cap \Omega) \to W^{1,2}(\R^d)$; that
    is, $(Eu)_{|U_x\cap\Omega} = u$ for every
    $u \in W^{1,2}(U_x\cap\Omega)$.
  \item The set $\Gamma$ is a $(d-1)$-set.
  \end{enumerate}
\end{assumption}

Recall that a Borel set $E \subset \R^d$ is an \emph{$N$-set}
or \emph{$N$-regular} if there exists a constant $c>0$ such that
\begin{equation}\label{eq:d-1}
  c r^{N} \leq \mathcal{H}_{N}\bigl(E \cap B_r(x)\bigr) \leq
  c^{-1} r^{N} \quad (x \in E,~r\leq 1)
\end{equation}
where $\mathcal{H}_N$ denotes the $N$-dimensional Hausdorff
measure. We refer to~\cite[Ch.~II.1]{JoWa84} for more details.

\begin{remark}
  \label{rem:assumption-extension-local}
  The regularity assumption on
  $\Gamma = \overline{\partial\Omega\setminus D}$ in
  \cref{assumption.s} is very mild.  The required
  Sobolev extension property is a deeply researched
  subject. Note that while there $D$ need only be closed, there
  is a condition on the the relative boundary $\partial D$ of
  $D$ within $\partial\Omega$, so the transition region between
  Dirichlet and Neumann boundary parts. Particular cases in
  which \cref{assumption.s} is satisfied include the
  one where there are Lipschitz charts available around
  $\Gamma$, or, more generally,
  when $\Omega$ is locally an $(\eps,\delta)$-domain around
  $\Gamma$. The latter is in fact
  optimal for $d=2$. We refer to~\cite[Section~6.4]{EgHDRe15}
  for more information.
\end{remark}

The immediate consequences of \cref{assumption.s}
needed in the following are as follows:
\begin{enumerate}[(a)]
\item There is a bounded linear extension
  operator 
  which extends both $\widetilde{W^{1,2}_D(\Omega)}$ and
  $W^{1,2}_D(\Omega)$ to
  $W^{1,2}_D(\R^d)$~(\cite[Theorem~6.9]{EgHDRe15}). In particular,
  these spaces coincide.
\item There is a well defined trace map
  $\tr \colon W^{1,2}_D (\Omega) \to
  L^\beta(\Gamma;\mathcal{H}_{d-1})$, where $\beta>2$
  (\cite{Bi09a}).
\end{enumerate}
Hence, for nonnegative
$b\in L^\infty(\Gamma;\mathcal{H}_{d-1})$, the form
$\formb \colon V\times V\to\C$ given by
\[
  \formb (u,v) \defn \forma (u,v) + \int_\Gamma b\,(\tr u)
  (\overline{\tr v}) \; d{\mathcal H}_{d-1} \quad (u,v\in V) ,
\]
is well defined, continuous, closed and accretive. In fact, it
is even a sub-Markovian form. The operator $B_2$ on
$L^2(\Omega)$ associated with this form is the negative
generator of an analytic contraction semigroup $S_2$ which
extends consistently to contraction semigroups $S_p$ on all
$L^p(\Omega)$-spaces, $p\in [1,\infty [$. The negative
generator of $S_p$ is denoted by $B_p$. All these properties
follow as in Section~\ref{sec.preliminaries}.

The operators $B_p$ are realizations of the second order
elliptic operator $-\dive (a \, \nabla\cdot )$ with mixed
Dirichlet and Robin boundary conditions. The corresponding
parabolic evolution problem associated with $B_p$ is formally
\begin{align*}
  u_t - \dive (a \, \nabla u ) & = f && \text{in } (0,\infty) \times \Omega , \\
  u & = 0 && \text{on } (0,\infty ) \times D , \\
  \langle a\nabla u , \nu\rangle + bu& = 0 && \text{on }  (0,\infty ) \times (\partial\Omega\setminus D) , \\
  u (0,\cdot )& = u_0 && \text{in } \Omega , 
\end{align*}
where $\nu$ is the unit outer normal. That is, one has
Dirichlet boundary conditions on $D$ and Robin boundary
conditions on $\partial\Omega\setminus D$, which reduce to
Neumann boundary conditions on the set $[b=0]$.

Since
\begin{equation*}
  \int_\Gamma b \, (\tr u) (\tr |u|^{p-2} \bar{u}) \; \ud
  {\mathcal H}_{d-1} = \int_\Gamma b \tr(|u|^p) \; \ud {\mathcal
    H}_{d-1} \geq 0
\end{equation*}
for every $u\in V\cap L^\infty(\Omega)$, by \cref{thm.numrange.neumann}, the numerical range of the operator
$B_p$ is contained in the same sector as the numerical range of the
operator $A_p$. From \cref{t-numrange-dissipative} and
the proof of \cref{cor.angle.analyticity}, we thus
obtain the following result. (Ultracontractivity is inferred
from the $W^{1,2}$-extension property of $V$, see
Remark~\ref{rem:domain-cusps-extension}, and
\cref{prop:ultracontract}.)

\begin{theorem}
  For every $p\in [2,\infty [$, the numerical range of $B_p$ is
  contained in the sector $\overline{\Sigma_{\theta_p}}$, where
  $\theta_p$ is as in \cref{t-numrange-dissipative}. Moreover, for every
  $p \in \mathopen]1,\infty [$,
  \begin{equation} \label{e-pazy-robin} \bigl\|(z-B_p )^{-1}
    \bigr\|_{\mathcal L(L^p(\Omega))} \le \frac{1}{\dist
      (z,\Sigma_{\theta_p} )}
  \end{equation}
  for every $z\in\C\setminus \overline{\Sigma_{\theta_p}}$ and
  the semigroup generated by $-B_p$ extends to an analytic
  contraction semigroup on the sector
  $\Sigma_{\frac{\pi}{2}-\theta_p}$ and is ultracontractive.
\end{theorem}

It is also possible to treat dynamical boundary conditions in
this approach. Fix a measurable subset $S\subseteq\Gamma$. Then
the embedding
$j \colon V \to L^2(\Omega) \times L^2(S;\mathcal{H}_{d-1})$
defined by $u\mapsto (u,\tr u)$ is continuous, injective and has
dense range, see~\cite[Lemma 2.10]{ElMeRe14}. Via this
embedding, the form $(\formb , V)$ induces also an operator
$\hat{B}_2$ on the Hilbert space
$H = L^2(\Omega) \times L^2(S;\mathcal{H}_{d-1})$. By~\cite[Proposition 2.16]{ElMeRe14}, the form $\formb$ is
sub-Markovian, so that $-\hat{B}_2$ generates an analytic
contraction semigroup $\hat{S}_2$ which extends consistently to
contraction semigroups $\hat{S}_p$ on all
$L^p(\Omega) \times L^p(S;\mathcal{H}_{d-1})$-spaces,
$p\in [1,\infty [$. The negative generator of $\hat{S}_p$ is
denoted by $\hat{B}_p$.

The corresponding parabolic evolution problem associated with
$\hat{B}_p$ is formally
\begin{align*}
  u_t -\dive (a \, \nabla u ) & = f && \text{in } (0,\infty)\times \Omega , \\
  u& = 0& & \text{on } (0,\infty)\times D , \\
  u_t + \langle a\nabla u , \nu\rangle + bu &= g& & \text{on } (0,\infty) \times S ,\\
  \langle a\nabla u , \nu\rangle + bu &= 0 && \text{on } (0,\infty) \times (\partial\Omega\setminus (D \cup S)) ,\\
  u (0,\cdot )& = u_0 && \text{in } \Omega ,
\end{align*}
that is, one has Dirichlet boundary conditions on $D$, dynamical
boundary conditions on $S$, and Robin boundary conditions on
$\partial\Omega\setminus (D\cup S)$, which reduce to Neumann
boundary conditions on the set $[b=0]$. Since $\hat{B}_p$ is
again fundamentally linked to the form $\forma$, the result
about the numerical range transfers immediately from
\cref{thm.numrange.neumann}. Regarding
ultracontractivity, we refer to continuity of the trace operator
$\tr \colon V \to L^\beta(S;\mathcal{H}_{d-1})$ where $\beta>2$
and the reasoning in~\cite[Lemma~2.19]{ElMeRe14}.

\begin{theorem}
  For every $p\in [2,\infty [$, the numerical range of
  $\hat{B}_p$ is contained in the sector
  $\overline{\Sigma_{\theta_p}}$, where $\theta_p$ is as in
  \cref{t-numrange-dissipative}. Moreover, for every
  $p \in \mathopen]1,\infty[$,
  \begin{equation} \label{e-pazy-dyn} \bigl\|(z-\hat{B}_p )^{-1}
    \bigr\|_{\mathcal L(L^p(\Omega))} \le \frac{1}{\dist
      (z,\Sigma_{\theta_p} )}
  \end{equation}
  for every $z\in\C\setminus \overline{\Sigma_{\theta_p}}$ and
  the semigroup generated by $-\hat{B}_p$ extends to an analytic
  contraction semigroup on the sector
  $\Sigma_{\frac{\pi}{2}-\theta_p}$ and is ultracontractive.
\end{theorem}

\begin{remark}
  \label{rem:dynbound-surface}
  It would also be possible to include a $(d-1)$-regular
  hyperplane $\Sigma \subset \Omega$ in $S$ in a straightforward
  manner. This would then lead to a dynamic ``jump condition''
  \begin{equation*}
    u_t + \bigl[\langle \forma\nabla u,\nu_\Sigma\rangle\bigr] = h \quad
    \text{on}~(0,\infty) \times \Sigma.
  \end{equation*}
  We refer to~\cite{ElMeRe14}.
\end{remark}

\section{Appendix: Intrinsic characterization for the form
  domain}

In this section, we give a completely \emph{intrinsic}
characterization for $V=W^{1,2}_D(\Omega)$, corresponding to the
philosophy in~\cite{Tr16}, see especially Remark~4 there. In
fact, we do so for the full scale $W^{1,p}_D (\Omega)$ with
$p \in \mathopen]1,\infty[$. We suppose that $\Omega$ is bounded
and let $p\in \mathopen]1,\infty[$ be fixed throughout this
section. The characterization is given under following very mild
assumption on $\partial\Omega$ and $D$ which we assume to hold
for the rest of this appendix:

\begin{assumption} \label{assumption.v3}
  \begin{enumerate}[(i)]
  \item For every point $x \in \partial D$, the relative
    boundary of $D$ within $\partial\Omega$, there is an open
    neighbourhood $U_x$ of $x$ such that $U_x \cap \Omega$ is
    connected and there exists a continuous linear extension
    operator
    $E \colon W^{1,p}(U_x \cap \Omega) \to W^{1,p}(\R^d)$; that
    is, $(Eu)_{|U_x\cap\Omega} = u$ for all
    $u\in W^{1,p}(U_x \cap \Omega)$.
  \item The boundary $\partial \Omega$ and the set $D$ itself
    are $(d-1)$-sets.
  \end{enumerate}
\end{assumption}


\begin{remark}\label{rem:intrinsic-assumption}
  Comparing to \cref{assumption.s}---which we do
  \emph{not} suppose to hold for this section---, the Sobolev
  extension condition is required \emph{only} on the relative
  boundary $\partial D$ of $D$ within $\partial\Omega$. Thus,
  the remaining part of $\partial\Omega\setminus \partial D$
  might be highly irregular in a topological sense. We do
  however suppose the measure-theoretic condition that
  $\partial \Omega$ and $D$ are $(d-1)$-regular in
  \cref{assumption.v3}, which is not included in
  \cref{assumption.s} and which effectively means that
  $\Gamma = \partial\Omega\setminus D$ is also $(d-1)$-regular.
\end{remark}

For 
closed $E \subseteq \partial\Omega$, we define the spaces
\[
  \uW^{1,p}_E (\Omega) \defn \bigl\{ u\in W^{1,p} (\Omega) \st
  \dist (\supp u , E) >0 \bigr\},
\]
and
\[
  C^{\infty}_E (\Omega) \defn \bigl\{ u \in C^\infty(\Omega)
  \st u = v|_\Omega~\text{for}~v\in C^\infty_c (\R^d ), \, \supp
  v\cap E = \emptyset \bigr\},
\]
and their closures in $W^{1,p} (\Omega)$:
\[
  W^{1,p}_E(\Omega) \defn \overline{\uW^{1,p}_E (\Omega
    )}^{W^{1,p}(\Omega)} \text{ and } \widetilde{W^{1,p}_E
    (\Omega)}\defn \overline{C^\infty_E (\Omega)}^{W^{1,p}} .
\]
We have already seen the latter two spaces in the previous
sections in the special case $p=2$. The characterization of
$W^{1,p}_D(\Omega)$ is as follows. (We use $\dist_D(x) \defn \dist(x,D)$.)

\begin{theorem}\label{thm:intrin-charact} Let 
  $u \in W^{1,p}(\Omega)$. The following are equivalent.
  \begin{enumerate}[(i)]
  \item\label{thm:intrin-charact-space}
    $u \in W^{1,p}_D(\Omega)$.
  \item\label{thm:intrin-charact-hardy}
    $u/\dist_D \in L^p(\Omega)$.
  \item\label{thm:intrin-charact-trace} For
    $\mathcal{H}_{d-1}$-almost every\ $x \in D$,
    \begin{equation*}
      \lim_{r\searrow0} \frac1{|B_r(x)|} \int_{B_r(x)\cap\Omega}
      |u| = 0.
    \end{equation*}
  \end{enumerate}
\end{theorem}

\begin{remark}
  \label{rem:hardy-inequality}
  If one and thus all of the conditions in
  \cref{thm:intrin-charact} hold true, then we have a
  \emph{Hardy inequality} for elements of $W^{1,p}_D(\Omega)$:
  \[
    \left( \int_\Omega \left| \frac{u}{\dist_D} \right|^p
    \right)^{\frac{1}{p}} \lesssim \| u\|_{W^{1,p} (\Omega)}
    \quad (u \in W^{1,p}_D(\Omega)).
  \]
  In particular,
  \begin{equation*}
    u \mapsto \| u\|_{W^{1,p}(\Omega)} + \left\| \frac{u}{\dist_D}
    \right\|_{L^p(\Omega)}
  \end{equation*}
  is an equivalent norm on
  $W^{1,p}_D(\Omega)$. 
\end{remark}

A consequence of the characterization of $W^{1,p}_D(\Omega)$ in
\cref{thm:intrin-charact} is that the constant one
function $\one$ is not an element of that space if $D\neq\emptyset$. The proof
follows after the one of \cref{thm:intrin-charact} below.

\begin{corollary}
  \label{cor:one-function}
  Let $\one \in W^{1,p}(\Omega)$ denote the constant one
  function. If $D\neq\emptyset$, then $\one \notin W^{1,p}_D(\Omega)$.
\end{corollary}

We next prove a preliminary geometric lemma which will allow us to
prove \cref{thm:intrin-charact} by reducing it to a
similar characterization theorem in a more regular situation,
\cref{prop:moritz-patrick-charact} below. It says
that a subset of a regular set can be extended to a regular set
in an arbitrarily small manner.  We state and prove it for a
general bounded $N$-regular set $\Lambda$. The proof relies on a
sort of dyadic decomposition for regular sets established by
David and refined by Christ and is given at the very end of the
paper.

\begin{lemma}
  \label{lem:regular-mantle}
  Let $\Lambda \subset \R^d$ be bounded and $N$-regular. Let
  further $\Xi \subseteq \Lambda$ and $\rho > 0$. Then there exists
  an $N$-set $\Xi^\bullet$ such that
  $\Xi \subseteq \Xi^\bullet \subseteq \Lambda$ and
  $\sup \{\dist(x,\Xi) \st x \in \Xi^\bullet\setminus\Xi\} \leq
  \rho$.
\end{lemma}

\begin{corollary}
  \label{cor:regular-mantle}
  There exists a closed $(d-1)$-set
  $\Upsilon \subseteq \partial\Omega$ such that
  $\dist(\Upsilon,D) > 0$ and for every point
  $x\in\overline{ \partial\Omega \setminus (D\cup \Upsilon )}$
  there is an open neighbourhood $U_x$ of $x$ such that
  $U_x \cap \Omega$ is connected and there exists a continuous
  linear extension operator
  $E \colon W^{1,p}(U_x \cap \Omega) \to W^{1,p}(\R^d)$.
\end{corollary}

\begin{proof}
  By \cref{assumption.v3}, for every
  $x \in \partial D$ there exists an open $W^{1,p}$-extension
  neighbourhood $U_x$ of $x$.  The family
  $(U_x)_{x\in \partial D}$ is then an open covering of
  $\partial D$. By compactness, it thus admits a finite
  subcovering $(U_{x_j})_j$.
  
  Now choose $\eps >0$ such that
  $\{x \in \R^d \st \dist (x,\partial D) < 3\eps \} \subseteq
  \bigcup_j U_{x_j}$, and set
  $C \defn \{ x\in\partial\Omega \setminus D \st \dist
  (x,\partial D) \geq 2\eps\}$. Clearly, for every
  $x \in \partial\Omega\setminus\overline{(D\cup C)}$ there is
  an open $W^{1,p}$-extension neighbourhood. Let $C^\bullet$ be a
  $(d-1)$ regular set containing $C$ with
  $\sup \{\dist(x,C) \st x \in C^\bullet\setminus C\} \leq \eps$
  and define $\Upsilon \defn \overline{C^\bullet}$; such a set
  exists by \cref{lem:regular-mantle}. Then $\Upsilon$ has
  the required properties; for the $(d-1)$ property,
  see~\cite[Ch.~VIII Proposition~1]{JoWa84}.
\end{proof}

With the foregoing result, we can now make use of the
characterisation of a zero trace property for a more regular
situation in~\cite{EgTo17} which we quote adapted to our
setting:

\begin{proposition}[{\cite[Theorem~2.1]{EgTo17}}]
  \label{prop:moritz-patrick-charact}
  Let $\Upsilon\subseteq\partial\Omega$ be a closed $(d-1)$-set
  such that for every point
  $x\in\overline{ \partial\Omega \setminus (D\cup \Upsilon )}$
  there is an open neighbourhood $U_x$ of $x$ such that
  $U_x \cap \Omega$ is connected and there exists a continuous
  linear extension operator
  $E \colon W^{1,p}(U_x \cap \Omega) \to W^{1,p}(\R^d)$.
  Let 
  $u \in W^{1,p}(\Omega)$. Then the following are equivalent:
  \begin{enumerate}[(i)]
  \item $u \in \widetilde{W^{1,p}_{D\cup\Upsilon}(\Omega)}$.
  \item $u/\dist_{D\cup\Upsilon} \in L^p(\Omega)$.
  \item For $\mathcal{H}_{d-1}$-a.e.\ $x \in D \cup \Upsilon$,
    \begin{equation*}
      \lim_{r\searrow0} \frac1{|B_r(x)|} \int_{B_r(x)\cap\Omega}
      |u| = 0.
    \end{equation*}
  \end{enumerate}
\end{proposition}


\begin{proof}[Proof of \cref{thm:intrin-charact}]
  Choose $\Upsilon\subseteq\partial\Omega$ as in
  \cref{cor:regular-mantle} and a cut-off function
  $\eta\in C^\infty_c (\R^d )$ such that
  $\supp \eta \cap \Upsilon = \emptyset$ and $\eta = 1$ in a
  neighbourhood of $D$. Write $u = (1-\eta)u + \eta u$. Clearly,
  $(1-\eta )u \in \uW^{1,p}_D(\Omega)$, so
  $(1-\eta)u/\dist_D \in L^p(\Omega)$ and
  \begin{equation*}
    \lim_{r\searrow0} \frac1{|B_r(x)|} \int_{B_r(x)\cap\Omega}
    |(1-\eta)u| = 0.
  \end{equation*}
  It is thus sufficient to prove all equivalences for $\eta u$
  only.
  
  (\ref{thm:intrin-charact-space}~$\implies$~\ref{thm:intrin-charact-hardy}):
  Let $\eta u\in W^{1,p}_D (\Omega
  )$. 
  Choose a sequence $(u_k) \subset \uW^{1,p}_D(\Omega)$
  approximating $u$ in $W^{1,p}(\Omega)$. We have
  $\eta u_k/\dist_{D\cup\Upsilon} \in L^p(\Omega)$ and
  $\eta u_k \to \eta u$ in
  $W^{1,p}(\Omega)$. \cref{prop:moritz-patrick-charact}
  implies that the set of $v \in W^{1,p}(\Omega)$ satisfying
  $v/\dist_{D\cup\Upsilon} \in L^p(\Omega)$ is \emph{closed} in
  $W^{1,p}(\Omega)$. Hence $\eta u/\dist_D \in L^p(\Omega)$.

  (\ref{thm:intrin-charact-hardy}~$\implies$~\ref{thm:intrin-charact-space}):
  Let $\eta u/\dist_D \in L^p(\Omega)$. Then
  $\eta u/\dist_{D\cup\Upsilon} \in L^p(\Omega)$ and
  $\eta u \in \widetilde{W^{1,p}_{D\cup\Upsilon}(\Omega)}$ by
  \cref{prop:moritz-patrick-charact}. In particular,
  $\eta u$ in $W^{1,p}(\Omega)$ can be approximated by a
  sequence of functions from
  $C_{D\cup\Upsilon}^\infty(\Omega)$. But
  $C_{D\cup\Upsilon}^\infty(\Omega) \subset
  \uW^{1,p}_D(\Omega)$, so $\eta u \in W^{1,p}(\Omega)$.
  
  (\ref{thm:intrin-charact-hardy}~$\iff$~\ref{thm:intrin-charact-trace}):
  Apply \cref{prop:moritz-patrick-charact} to
  $\eta u$.
\end{proof}

\begin{proof}[Proof of \cref{cor:one-function}]
  Suppose that $\one \in W^{1,p}_D(\Omega)$. Then, by
  \cref{thm:intrin-charact},
  \begin{equation}
    \lim_{r\searrow0} \frac1{|B(y,r)|} \int_{B(r,y) \cap \Omega}
    \one = \lim_{r\searrow0} \frac{|B(y,r) \cap
      \Omega|}{|B(y,r)|} = 0\label{eq:constant-one-trace}
  \end{equation}
  for $\mathcal{H}_{d-1}$-almost every $y \in D$. We will show
  that this leads to a contradiction.
  
  Let $x \in \partial D$, the relative boundary of $D$ within
  $\partial\Omega$. By \cref{assumption.v3}, there
  exists an open neighbourhood $U_x$ of $x$ such that
  $U_x \cap \Omega$ has the $W^{1,p}$-extension property. A
  domain with the $W^{1,p}$-extension property is necessarily
  $d$-regular by a fundamental result by Haj{\l}asz, Koskela and
  Tuominen~\cite{HjKoTu08}, so there is a constant $c>0$ such
  that
  \begin{equation*}
    |B(y,r) \cap \Omega \cap U_x| \geq c r^d \quad (y \in
    \Omega \cap U_x,~r\leq 1).
  \end{equation*}
  This property also holds for
  $y \in \partial\Omega \cap U_x \subset \partial(\Omega \cap
  U_x)$. Indeed, let $r\leq1$ and choose
  $z \in B(y,r/2) \cap \Omega \cap U_x$. Then
  $B(z,r/2) \cap \Omega \cap U_x \subset B(y,r) \cap \Omega\cap
  U_x$, hence
  \begin{equation}\label{eq:meas-dens-bound}
    |B(y,r) \cap \Omega \cap U_x|  \geq c2^{-d}r^d \quad (y \in
    \partial\Omega \cap U_x,~r\leq 1).
  \end{equation}

  Now let $0 < r_0 \leq 1$ be such that $B(x,2r_0) \subset
  U_x$. Consider $y \in B(x,r_0) \cap D$.  Then
  $B(y,r_0) \subset U_x$, so there is a constant $c_0>0$ such
  that
  \begin{equation*}
    \frac{|B(y,r) \cap
      \Omega|}{|B(y,r)|} = \frac{|B(y,r) \cap
      \Omega \cap U_x|}{|B(y,r)|} \geq c_0 > 0 
  \end{equation*}
  for all $r \leq r_0$
  by~\eqref{eq:meas-dens-bound}. By~\eqref{eq:constant-one-trace},
  this is only possible if
  $\mathcal{H}_{d-1}(B(x,r_0) \cap D) = 0$. But $D$ is
  $(d-1)$-regular, so $\mathcal{H}_{d-1}(B(x,r_0)\cap D) >
  0$. This is the contradiction.
\end{proof}

\subsection*{Proof of \cref{lem:regular-mantle}}

In this final subsection, we prove
\cref{lem:regular-mantle}. As already mentioned above, the
proof relies on the following Christ decomposition for regular
sets:

\begin{theorem}[{\cite[Theorem~11]{MichaelChrist1990}}]
  \label{thm:dyadic-regular}
  Let $\Lambda \subset \R^d$ be bounded and $N$-regular. Then
  there exists a collection of open sets
  $\{Q_\alpha^k \subseteq \Lambda \st k \in \N_0,~\alpha \in
  I_k\}$, where $I_k$ is an index set for every $k \in \N_0$,
  and constants $\delta\in\mathopen]0,1[$, $a_0 >0$,
  $c_1 < \infty$ such that the following hold true:
  \begin{enumerate}[(i)]
  \item\label{thm:dyadic-regular-exhaust}
    $\mathcal{H}_{N}\bigl(\Lambda \setminus \bigcup_{\alpha\in
      I_k} Q_\alpha^k\bigr) = 0$ for every $k \in \N_0$,
  \item if $\ell,k \in \N_0$ and $\ell \geq k$, then for every
    $\alpha \in I_k$ and $\beta \in I_\ell$, either
    $Q_\beta^\ell \cap Q_\alpha^k = \emptyset$ or
    $Q_\beta^\ell \subseteq
    Q_\alpha^k$,\label{thm:dyadic-regular-hierarchy}
  \item if $k \in \N_0$ and $\alpha,\beta \in I_k$, then
    $Q_\beta^k \cap Q_\alpha^k =
    \emptyset$,\label{thm:dyadic-regular-hierarchy-disjoint}
  \item for every $k \in \N_0$ and $\alpha \in I_k$, there holds
    $\diam(Q_\alpha^k) \leq c_1
    \delta^k$,\label{thm:dyadic-regular-size}
  \item for every $k \in \N_0$ and $\alpha \in I_k$, there is
    $z_\alpha^k \in \Lambda$ such that
    $B(z_\alpha^k,a_0\delta^k) \cap \Lambda \subseteq
    Q_\alpha^k$.\label{thm:dyadic-regular-fat}
  \end{enumerate}
\end{theorem}
\begin{remark}
  It will be useful to observe that by
  property~\ref{thm:dyadic-regular-fat} of the Christ
  decomposition in \cref{thm:dyadic-regular} and
  $N$-regularity of $\Lambda$, there is a constant
  $c>0$ such that
  \begin{equation}
    \label{eq:dyadic-regular}
    \mathcal{H}_N(Q_\alpha^k) \geq
    \mathcal{H}_N\bigl(B(z_\alpha^k,a_0\delta^k) \cap
    \Lambda\bigr) \geq c(a_0\delta^k)^N \quad 
    (\alpha \in I_k)
  \end{equation}
  for all $k \in \N_0$ such that $a_0\delta^k \leq
  1$.\label{rem:dyadic-regular-useful}
\end{remark}

In fact, the Christ decomposition as
established in~\cite[Theorem~11]{MichaelChrist1990} has some more
properties and is valid for \emph{locally doubling metric
  measure spaces} in general; the ``locally'' part is due to
Morris~\cite[Proposition~4.2]{Morris_2012}. We have just extracted the
necessary bits needed to prove \cref{lem:regular-mantle}
which we repeat:

\smallskip \emph{Let $\Lambda \subset \R^d$ be bounded and $N$-regular. Let
  further $\Xi \subseteq \Lambda$ and $\rho > 0$. Then there exists
  an $N$-set $\Xi^\bullet$ such that
  $\Xi \subseteq \Xi^\bullet \subseteq \Lambda$ and
  $\sup \{\dist(x,\Xi) \st x \in \Xi^\bullet\setminus\Xi\} \leq
  \rho$.}

\begin{proof}[Proof of \cref{lem:regular-mantle}]
  Consider the Christ decomposition of $\Lambda$ and its data as
  stated in \cref{thm:dyadic-regular}.  Let $M \in \N_0$
  be so large that $c_1 \delta^M \leq \rho\wedge1$ and define
  \begin{equation*}
    \Xi^\bullet \defn \Xi \cup \bigl\{ Q_\alpha^M \st \alpha \in
    I_M,~Q_\alpha^M\cap \Xi \neq \emptyset\bigr\}.
  \end{equation*}
  By the choice of $M$ and
  property~\ref{thm:dyadic-regular-size} of the Christ
  decomposition, we already have
  $\sup \{\dist(x,\Xi) \st x \in \Xi^\bullet\setminus\Xi\} \leq
  \rho$. We show that $\Xi^\bullet$ is $N$-regular. Since
  $\Xi^\bullet \subseteq \Lambda$ and the latter is $N$-regular,
  the upper estimate
  $\mathcal H_N(\Xi^\bullet \cap B(x,r)) \lesssim r^N$ for all
  $x \in \Xi^\bullet$ and $r \leq 1$ as in~\eqref{eq:d-1} is for
  free.

  For the lower estimate, let $x \in \Xi^\bullet$. If there is
  $\alpha \in I_M$ such that $x \in Q_\alpha^M$, then
  $Q_\alpha^M \subseteq \Xi^\bullet$ due to
  property~\ref{thm:dyadic-regular-hierarchy-disjoint} of the
  Christ decomposition. Thus it is enough to show the lower
  estimate for $N$-regularity for $Q_\alpha^M$. (Of course, we
  thereby prove that $Q_\alpha^M$ is $N$-regular since the upper
  estimate is again for free.)

  In fact, we can assume even that $x$ is an element of a member
  of \emph{every} generation $k \in \N_0$, i.e.,
  \begin{equation*}
    x \in \Upsilon \defn
    \bigl\{ y \in \Lambda\st \text{for every}~k \in \N_0~\text{there
      is}~\alpha\in I_k~\text{such that}~y \in Q_\alpha^k\bigr\}.
  \end{equation*}
  Indeed, suppose we want to show a lower $N$-regularity
  estimate for $z \in \Lambda\setminus\Upsilon$. The set
  $\Lambda\setminus\Upsilon$ is of $\mathcal{H}_N$-measure zero
  by property~\ref{thm:dyadic-regular-exhaust} of the Christ
  decomposition. For every $r \leq 1$, the intersections
  $\Lambda \cap B(z,r/2)$ have positive $\mathcal{H}_N$-measure
  by $N$-regularity of $\Lambda$, hence they cannot be subsets
  of $\Lambda\setminus\Upsilon$. This implies that for every
  $r \leq 1$, there is $x \in \Upsilon \cap B(z,r/2)$. Since
  $B(z,r)$ contains $B(x,r/2)$, it is thus enough to prove the
  lower estimate for $N$-regularity for $x \in \Upsilon$.

  So, let $x \in \Upsilon$ and let $\alpha \in I_M$ be such that
  $x \in Q_\alpha^M$.

  First suppose that $c_1\delta^M<r\leq1$. Then, by
  property~\ref{thm:dyadic-regular-size} of the Christ
  decomposition, $Q_\alpha^M \cap B(x,r) = Q_\alpha^M$. So,
  using~\eqref{eq:dyadic-regular},
  \begin{align*}
    \mathcal{H}_N\bigl(Q_\alpha^M \cap B(x,r)\bigr)   =
    \mathcal{H}_N(Q_\alpha^M) \geq c (a_0\delta^M)^N \geq
    c(a_0\delta^M)^Nr^N. 
  \end{align*}

  Next, let $r \leq c_1\delta^M$. Choose $\ell \in \N_0$ such
  that $c_1\delta^{\ell}\leq r \leq
  c_1\delta^{\ell-1}$. Clearly, $\ell > M$. Due to
  $x\in\Upsilon$, there exists $\beta \in I_\ell$ such that
  $x \in Q_\beta^\ell$ and we have
  $Q_\beta^\ell \subseteq Q_\alpha^M$ by
  property~\ref{thm:dyadic-regular-hierarchy} of the Christ
  decomposition. Again, from
  property~\ref{thm:dyadic-regular-size} of the Christ
  decomposition and the choice of $\ell$ we have
  $Q_\beta^\ell \cap B(x,r) =
  Q_\beta^\ell$. Using~\eqref{eq:dyadic-regular} and the choice
  of $\ell$,
  \begin{align*}
    \mathcal{H}_N\bigl(Q_\alpha^M \cap B(x,r)\bigr) &\geq 
    \mathcal{H}_N\bigl(Q_\beta^\ell \cap B(x,r)\bigr)   =
    \mathcal{H}_N(Q_\beta^\ell) \\ &\geq c (a_0\delta^\ell)^N \geq
    c(a_0\delta/c_1)^Nr^N. 
  \end{align*}
  This completes the proof.
\end{proof}

\medskip

\textbf{Acknowledgment.} We wish to thank Pertti Mattila
(University of Hel\-sin\-ki) and Moritz Egert (Universit\'{e}
Paris-Sud) for pointing out the Christ decomposition and ideas
for the proof of \cref{lem:regular-mantle}.

\providecommand{\bysame}{\leavevmode\hbox
  to3em{\hrulefill}\thinspace}

\def\MR#1{}

\bibliographystyle{amsplain} 
\bibliography{UniformResolvent}

\providecommand{\bysame}{\leavevmode\hbox to3em{\hrulefill}\thinspace}
\providecommand{\MR}{\relax\ifhmode\unskip\space\fi MR }
\providecommand{\MRhref}[2]{%
  \href{http://www.ams.org/mathscinet-getitem?mr=#1}{#2}
}
\providecommand{\href}[2]{#2}
\begin{thebibliography}{10}

\bibitem{Am95}
Herbert Amann, \emph{Linear and quasilinear parabolic problems. {V}ol. {I}},
  Monographs in Mathematics, vol.~89, Birkh\"{a}user Boston, Inc., Boston, MA,
  1995, Abstract linear theory. \MR{1345385}

\bibitem{Ar04}
W.~Arendt, \emph{Semigroups and evolution equations: functional calculus,
  regularity and kernel estimates}, Handbook of Differential Equations (C. M.
  Dafermos, E. Feireisl eds.), Elsevier/North Holland, 2004, pp.~1--85.

\bibitem{Bi09a}
Markus Biegert, \emph{On traces of {S}obolev functions on the boundary of
  extension domains}, Proc. Amer. Math. Soc. \textbf{137} (2009), no.~12,
  4169--4176. \MR{2538577 (2010m:46045)}

\bibitem{BMMM14}
Kevin Brewster, Dorina Mitrea, Irina Mitrea, and Marius Mitrea, \emph{Extending
  {S}obolev functions with partially vanishing traces from locally
  {$(\varepsilon,\delta)$}-domains and applications to mixed boundary
  problems}, J. Funct. Anal. \textbf{266} (2014), no.~7, 4314--4421.
  \MR{3170211}

\bibitem{CaDr19}
Andrea Carbonaro and Oliver Dragi\v{c}evi\'{c}, \emph{Convexity of power
  functions and bilinear embedding for divergence-form operators with complex
  coefficients}, J. Eur. Math. Soc. (2019), no.~to appear. \MR{3626567}

\bibitem{CFMP05}
R.~Chill, E.~Fa\v{s}angov\'a, G.~Metafune, and D.~Pallara, \emph{The sector of
  analyticity of the {O}rnstein-{U}hlenbeck semigroup in {$L^p$} spaces with
  respect to invariant measure}, J. London Math. Soc. \textbf{71} (2005),
  703--722.

\bibitem{CFMP06}
\bysame, \emph{The sector of analyticity of nonsymmetric submarkovian
  semigroups generated by elliptic operators}, C. R. Acad. Sci. Paris
  \textbf{342} (2006), 909--914.

\bibitem{MichaelChrist1990}
Michael Christ, \emph{A {$T(b)$} theorem with remarks on analytic capacity and
  the cauchy integral}, Colloquium Mathematicae \textbf{60-61} (1990), no.~2,
  601--628 (eng).

\bibitem{CiMa05}
A.~Cialdea and V.~Maz'ya, \emph{Criterion for the {$L^p$}-dissipativity of
  second order differential operators with complex coefficients}, J. Math.
  Pures Appl. (9) \textbf{84} (2005), no.~8, 1067--1100. \MR{2155899}

\bibitem{CoWe77}
Ronald~R. Coifman and Guido Weiss, \emph{Extensions of {H}ardy spaces and their
  use in analysis}, Bull. Amer. Math. Soc. \textbf{83} (1977), no.~4, 569--645.
  \MR{0447954 (56 \#6264)}

\bibitem{Dan02}
D.~Daners, \emph{A priori estimates for solutions to elliptic equations on
  non-smooth domains}, Proc. Roy. Soc. Edinburgh Sect. A \textbf{132} (2002),
  793--813.

\bibitem{Da89}
E.~B. Davies, \emph{Heat kernels and spectral theory}, Cambridge Tracts in
  Mathematics, vol.~92, Cambridge University Press, Cambridge, 1989.
  \MR{990239}

\bibitem{DeHiPr03}
R.~Denk, M.~Hieber, and J.~Pr\"uss, \emph{{${\mathcal R}$}-{B}oundedness,
  {F}ourier {M}ultipliers and {P}roblems of {E}lliptic and {P}arabolic {T}ype},
  Memoirs Amer. Math. Soc., vol. 166, Amer. Math. Soc., Providence, R.I., 2003.

\bibitem{EgHDRe15}
Moritz Egert, Robert Haller-Dintelmann, and Joachim Rehberg, \emph{Hardy's
  inequality for functions vanishing on a part of the boundary}, Potential
  Anal. \textbf{43} (2015), no.~1, 49--78. \MR{3361789}

\bibitem{EgTo17}
Moritz Egert and Patrick Tolksdorf, \emph{Characterizations of {S}obolev
  functions that vanish on a part of the boundary}, Discrete Contin. Dyn. Syst.
  Ser. S \textbf{10} (2017), no.~4, 729--743. \MR{3640535}

\bibitem{GrKaRe01}
J.~A. Griepentrog, H.-C. Kaiser, and J.~Rehberg, \emph{Heat kernel and
  resolvent properties for second order elliptic differential operators with
  general boundary conditions on {$L^p$}}, Adv. Math. Sci. Appl. \textbf{11}
  (2001), no.~1, 87--112. \MR{1841562}

\bibitem{HjKoTu08}
Piotr Haj{\l}asz, Pekka Koskela, and Heli Tuominen, \emph{Sobolev embeddings,
  extensions and measure density condition}, J. Funct. Anal. \textbf{254}
  (2008), no.~5, 1217--1234. \MR{2386936}

\bibitem{JoWa84}
Alf Jonsson and Hans Wallin, \emph{Function spaces on subsets of {${\mathbf
  R}^n$}}, Math. Rep. \textbf{2} (1984), no.~1, xiv+221. \MR{820626}

\bibitem{KaWe01}
N.~Kalton and L.~Weis, \emph{The {$H^\infty$}-calculus and sums of closed
  operators}, Math. Ann. \textbf{321} (2001), 319--345.

\bibitem{Km02}
P.~C. Kunstmann, \emph{{$L\sb p$}-spectral properties of the {N}eumann
  {L}aplacian on horns, comets and stars}, Math. Z. \textbf{242}, 183--201.

\bibitem{KuWe04}
P.~C. Kunstmann and L.~Weis, \emph{Maximal {$L^p$} regularity for parabolic
  equations, {F}ourier multiplier theorems and {$H^\infty$} functional
  calculus}, Levico Lectures, Proceedings of the Autumn School on Evolution
  Equations and Semigroups (M. Iannelli, R. Nagel, S. Piazzera eds.), vol.~69,
  Springer Verlag, Heidelberg, Berlin, 2004, pp.~65--320.

\bibitem{La87}
Damien Lamberton, \emph{\'{E}quations d'\'evolution lin\'eaires associ\'ees \`a
  des semi-groupes de contractions dans les espaces {$L^p$}}, J. Funct. Anal.
  \textbf{72} (1987), no.~2, 252--262. \MR{886813 (88g:47085)}

\bibitem{Lu95}
A.~Lunardi, \emph{Analytic {S}emigroups and {O}ptimal {R}egularity in
  {P}arabolic {P}roblems}, Progress in Nonlinear Differential Equations and
  Their Applications, vol.~16, Birkh\"auser, Basel, 1995.

\bibitem{MzPo06}
V.~G. Maz'ya and S.~V. Poborchi\u{\i}, \emph{Theorems for embedding {S}obolev
  spaces on domains with a peak and on {H}\"{o}lder domains}, Algebra i Analiz
  \textbf{18} (2006), no.~4, 95--126. \MR{2262585}

\bibitem{Mz11}
Vladimir Maz'ya, \emph{Sobolev spaces with applications to elliptic partial
  differential equations}, augmented ed., Grundlehren der Mathematischen
  Wissenschaften [Fundamental Principles of Mathematical Sciences], vol. 342,
  Springer, Heidelberg, 2011. \MR{2777530 (2012a:46056)}

\bibitem{Morris_2012}
Andrew~J. Morris, \emph{The {K}ato square root problem on submanifolds},
  Journal of the London Mathematical Society \textbf{86} (2012), no.~3,
  879--910.

\bibitem{Ou04}
E.~M. Ouhabaz, \emph{Analysis of {H}eat {E}quations on {D}omains}, London
  Mathematical Society Monographs, vol.~30, Princeton University Press,
  Princeton, 2004.

\bibitem{Pa83}
A.~Pazy, \emph{Semigroups of {L}inear {O}perators and {A}pplications to
  {P}artial {D}ifferential {E}quations}, Applied Mathematical Sciences,
  vol.~44, Berlin, 1983.

\bibitem{PrSi16}
Jan Pr\"{u}ss and Gieri Simonett, \emph{Moving interfaces and quasilinear
  parabolic evolution equations}, Monographs in Mathematics, vol. 105,
  Birkh\"{a}user/Springer, [Cham], 2016. \MR{3524106}

\bibitem{EHRT19}
A.~F.~M. ter Elst, R.~Haller-Dintelmann, J.~Rehberg, and P.~Tolksdorf, \emph{On
  the {$L^p$}-theory for second-order elliptic operators in divergence form
  with complex coefficients}, Preprint (2019),
  https://arxiv.org/abs/1903.06692. \MR{3979937}

\bibitem{ElMeRe14}
A.~F.~M. ter Elst, M.~Meyries, and J.~Rehberg, \emph{Parabolic equations with
  dynamical boundary conditions and source terms on interfaces}, Ann. Mat. Pura
  Appl. (4) \textbf{193} (2014), no.~5, 1295--1318. \MR{3262633}

\bibitem{Tr16}
Hans Triebel, \emph{A note on function spaces in rough domains}, Tr. Mat. Inst.
  Steklova \textbf{293} (2016), no.~Funktsional. Prostranstva, Teor. Priblizh.,
  Smezhnye Razdely Mat. Anal., 346--351, English version published in Proc.
  Steklov Inst. Math. {{\textbf{293}}} (2016), no. 1, 338--342. \MR{3628489}

\bibitem{Uk19}
Alexander Ukhlov, \emph{Extension operators on {S}obolev spaces with decreasing
  integrability}, Preprint (2019), https://arxiv.org/abs/1908.09322.
  \MR{3979937}

\end{thebibliography}

\end{document}